\documentclass[11pt]{article}
\usepackage[a4paper,margin=0.87in]{geometry}
\usepackage{indentfirst}
\usepackage{amsmath,amssymb,amsthm,mathtools,microtype}
\usepackage[colorlinks=true,allcolors=black]{hyperref}
\usepackage[nameinlink,noabbrev]{cleveref}
\hypersetup{pdftitle={(alpha, c)-spaceability of continuous functions with prescribed Hausdorff graph dimension},pdfauthor={Daniel M. Pellegrino; Anselmo Raposo Jr.}}
\newtheorem{maintheorem}{Theorem}

\crefname{maintheorem}{theorem}{theorems}
\Crefname{maintheorem}{Theorem}{Theorems}
\newtheorem{lemma}{Lemma}[section]
\newcommand{\R}{\mathbb R}
\newcommand{\N}{\mathbb N}
\newcommand{\dimH}{\dim_{\rm H}}
\newcommand{\cc}{\mathfrak c}
\newcommand{\Hs}{H_s}
\title{$(\alpha,\cc)$-spaceability of continuous functions\\with prescribed Hausdorff graph dimension}
\author{Daniel M. Pellegrino\textsuperscript{1}\qquad Anselmo Raposo Jr.\textsuperscript{2}\\[0.6em]
\small\textsuperscript{1} Departamento de Matem\'atica, Universidade Federal da Para\'iba, Jo\~ao Pessoa, Brazil\\
\small\textsuperscript{2} Coordena\c{c}\~ao do Curso de Matem\'atica -- Bacharelado,\\
\small Universidade Federal do Maranh\~ao, S\~ao Lu\'is, Brazil\\[0.4em]
\small\texttt{daniel.pellegrino@academico.ufpb.br}\\
\small\texttt{anselmo.junior@ufma.br}}
\date{}
\begin{document}
\maketitle
\begin{abstract}
Let $1<s\le2$ and let $\Hs[0,1]$ consist of the real-valued continuous functions on $[0,1]$ whose graphs have Hausdorff dimension $s$. We prove that every finite-dimensional subspace $E\subset C[0,1]$ with $E\setminus\{0\}\subset\Hs[0,1]$ is contained in a closed complemented subspace isomorphic to $c_0$, all of whose nonzero elements still belong to $\Hs[0,1]$. This settles a question of Liu, Shi and Zhang. We also obtain the exact criterion: $\Hs[0,1]$ is $(\alpha,\cc)$-spaceable precisely when $\alpha$ is finite. We also show that every admissible finite-dimensional subspace is contained in a dense subspace of Hamel dimension $\cc$ whose nonzero elements have the same prescribed graph dimension. This also recovers maximal pointwise dense-lineability of $\Hs[0,1]$.
\end{abstract}

\section{Introduction}
Lineability studies large vector subspaces inside sets that need not themselves be vector spaces. A subset $A$ of a topological vector space $X$ is \emph{lineable} if $A\cup\{0\}$ contains an infinite-dimensional vector subspace. It is \emph{spaceable} if such a subspace can be chosen closed, and \emph{dense-lineable} if it can be chosen dense in $X$. The exclusion of the zero vector is essential when the defining property of $A$ is incompatible with zero. Spaceability and dense-lineability thus impose distinct topological requirements on the linear structures under consideration.

One of the earliest results in this direction is due to Gurariy \cite{Gurariy1966}, who proved in 1966 that the continuous nowhere differentiable functions on $[0,1]$, together with zero, contain an infinite-dimensional vector subspace. The seminal work of Aron, Gurariy and Seoane-Sep\'ulveda \cite{AGS2005} helped establish the systematic study of lineability and spaceability for classes of functions with exceptional properties. These developments show that irregularity need not prevent a rich linear structure.

Stronger notions ask whether a prescribed initial subspace can always be enlarged while retaining the same defining property. We write $\N=\{1,2,\ldots\}$, $\aleph_0=|\N|$, and $\cc=|\R|$; all vector-space dimensions below are Hamel dimensions, namely the cardinalities of algebraic bases.

Given cardinals $1\le\alpha\le\beta\le\dim X$, a set $A\subset X$ is \emph{$(\alpha,\beta)$-lineable} if there exists an $\alpha$-dimensional subspace $E\subset X$ with $E\setminus\{0\}\subset A$, and every such $E$ is contained in a $\beta$-dimensional subspace $F\subset X$ satisfying $F\setminus\{0\}\subset A$. If $F$ can always be chosen closed, $A$ is \emph{$(\alpha,\beta)$-spaceable}; if it can always be chosen dense in $X$, $A$ is \emph{$(\alpha,\beta)$-dense-lineable}. The existence condition on $E$ is retained in both variants. This extension approach was introduced in \cite{FPT2020} and further developed in \cite{FPRR2024}. We allow $\alpha=\beta$ in the definitions used here.

A set $A\subset X$ is \emph{$\beta$-dense-lineable} if there is a dense $\beta$-dimensional subspace $M\subset X$ such that $M\setminus\{0\}\subset A$. A nonempty set $A\subset X$ is \emph{pointwise $\beta$-lineable} if, for every $f\in A$, there exists a $\beta$-dimensional subspace $M_f\subset X$ with $f\in M_f$ and $M_f\setminus\{0\}\subset A$. If each $M_f$ can be chosen closed or dense in $X$, the corresponding properties are \emph{pointwise $\beta$-spaceability} and \emph{pointwise $\beta$-dense-lineability}, respectively. The pointwise approach was introduced in \cite{PR2021}; see also \cite{CMGP2024} for further developments concerning pointwise dense-lineability. We use \emph{maximal} when $\beta=\dim X$. For a nonempty set $A$ with $0\notin A$ that is stable under multiplication by nonzero scalars, pointwise $\beta$-dense-lineability is equivalent to $(1,\beta)$-dense-lineability: every $f\in A$ generates an admissible line, and every admissible line contains such a nonzero element. Requiring extension of every initial subspace can produce obstructions that are absent from ordinary lineability or spaceability.

Here we investigate this distinction for continuous functions with prescribed Hausdorff graph dimension. For a nonempty compact set $K\subset\R$, let $C(K)$ be the real Banach space of continuous functions on $K$, with norm $\|f\|_\infty=\sup_{x\in K}|f(x)|$. If $f$ is defined on $D\subset\R$, write
\[G_f(D)=\{(x,f(x)):x\in D\}.\]
For a set $B$ in a Euclidean space and $t>0$, we use the normalization
\[
\mathcal H^t(B)=\lim_{\delta\downarrow0}\inf\left\{
\sum_{i=1}^{\infty}(\operatorname{diam}U_i)^t:
B\subset\bigcup_{i=1}^{\infty}U_i,\quad
\operatorname{diam}U_i\le\delta\right\},
\]
where the infimum runs over countable covers and diameter is measured with the Euclidean metric. The \emph{Hausdorff dimension} is
\[\dimH B=\inf\{t>0:\mathcal H^t(B)=0\},\]
with $\inf\varnothing=+\infty$. In particular, this normalization gives $\mathcal H^1([0,1])=1$. For $1<s\le2$, define
\[\Hs[0,1]=\{f\in C[0,1]:\dimH G_f([0,1])=s\}.\]

For comparison with earlier work, we also recall the box dimensions. For a nonempty bounded set $B$ in a Euclidean space, let $N_\delta(B)$ be the smallest number of sets of diameter at most $\delta$ needed to cover $B$. Its \emph{lower} and \emph{upper box dimensions} are
\[
\underline{\dim}_{\rm B}B=\liminf_{\delta\downarrow0}\frac{\log N_\delta(B)}{-\log\delta},
\qquad
\overline{\dim}_{\rm B}B=\limsup_{\delta\downarrow0}\frac{\log N_\delta(B)}{-\log\delta}.
\]
When these agree, their common value is the \emph{box dimension}; see also \cite[Section 2]{LSZ2026}.

Spaceability of $\Hs[0,1]$ was established in \cite{BMPS2021}; see also the spaceability and latticeability result of \cite[Theorem 2.8]{CFST2022}. These results guarantee the existence of closed infinite-dimensional subspaces, while leaving open the extension of arbitrary admissible initial subspaces. Bonilla and collaborators \cite[Theorem 2.6]{BMPS2021} also establish maximal dense-lineability for continuous functions whose graphs have both Hausdorff and box dimensions equal to $s$. Related work \cite{LZS2025} studied algebraic genericity for continuous functions with prescribed fractal graph dimensions. Liu, Shi and Zhang \cite{LSZ2026} subsequently proved that $\Hs[0,1]$ is $(p,\cc)$-spaceable for $p=1,2$ and $(n,n+m)$-lineable for every $m,n\in\N$. They asked whether $(p,\cc)$-spaceability holds for all integers $p\ge3$; see \cite[Question 3.19]{LSZ2026}.

We answer this question by extending every admissible finite-dimensional subspace to a closed complemented subspace isomorphic to
\[
c_0=\{a=(a_j)_{j\ge1}\in\R^{\N}:a_j\to0\},
\qquad \|a\|_\infty=\sup_{j\ge1}|a_j|.
\]
Here $F\simeq c_0$ means that there is a bounded linear bijection from $F$ onto $c_0$ with bounded inverse. A closed subspace $F$ of a Banach space $X$ is \emph{complemented} if it is the range of a bounded linear projection $P:X\to X$, that is, $P^2=P$.

\begin{maintheorem}\label{thm:main}
Let $1<s\le2$ and let $E\subset C[0,1]$ be finite-dimensional, with $E\setminus\{0\}\subset\Hs[0,1]$. Then there is a closed complemented subspace $F\subset C[0,1]$ such that
\[E\subset F,\qquad F\simeq c_0,\qquad F\setminus\{0\}\subset\Hs[0,1].\]
\end{maintheorem}

\begin{maintheorem}\label{thm:threshold}
Let $1<s\le2$. For every cardinal $1\le\alpha\le\cc$,
\[\Hs[0,1]\text{ is }(\alpha,\cc)\text{-spaceable}
\quad\Longleftrightarrow\quad \alpha<\aleph_0.\]
\end{maintheorem}

\Cref{thm:main} supplies the positive extension result for all finite initial dimensions. The obstruction for infinite initial dimensions is already known from \cite[Corollary 2.11]{LSZ2026}. We give a direct proof, using the maximal dense-lineability established in \cite[Theorem 2.6]{BMPS2021}, that also covers $\alpha=\cc$ under our convention allowing equal cardinals. Thus \Cref{thm:threshold} identifies finiteness of the initial dimension as the exact threshold for closed extensions.

The proof of \Cref{thm:main} uses the localization in \cite[Proposition 3.10]{LSZ2026}, which preserves the graph dimension of every nonzero element of $E$ outside a compact set $K$ of dimension $s-1$. We construct a closed copy of $c_0$ in $C(K)$ whose nonzero elements have graph dimension $s$. Each coordinate is distributed over infinitely many disjoint pieces of $K$ with dimensions tending to $s-1$, ensuring that nonzero infinite combinations retain the required graph dimension. Affine extension across the gaps of $K$ then allows this space to be added to $E$. Since $C[0,1]$ is separable, a closed subspace isomorphic to $c_0$ is complemented by Sobczyk's theorem \cite{Sobczyk}; we also construct a projection directly.

The same finite-dimensional extension leads to dense subspaces. Hausdorff graph dimension is unchanged by adding a polynomial, and polynomials are dense in $C[0,1]$. Perturbing selected algebraic basis vectors of the space supplied by \Cref{thm:main} therefore yields a dense extension while keeping the prescribed initial subspace fixed. The pointwise consequence already follows from the $(1,\cc)$-spaceability in \cite{LSZ2026} and the general criterion in \cite[Theorem 2.3]{CMGP2024}, applied with the nonzero polynomials as the auxiliary set. The result below extends the dense conclusion to every admissible finite-dimensional initial subspace.

\begin{maintheorem}\label{thm:dense}
Let $1<s\le2$ and let $E\subset C[0,1]$ be finite-dimensional, with $E\setminus\{0\}\subset\Hs[0,1]$. Then there is a dense vector subspace $M\subset C[0,1]$ such that
\[E\subset M,\qquad \dim M=\cc,\qquad M\setminus\{0\}\subset\Hs[0,1].\]
Consequently, $\Hs[0,1]$ is $(n,\cc)$-dense-lineable for every $n\in\N$ and maximal pointwise dense-lineable in $C[0,1]$.
\end{maintheorem}

Section~\ref{sec:compact} develops the construction on a compact set, and Section~\ref{sec:spaceable} proves \Cref{thm:main,thm:threshold}. Section~\ref{sec:dense} then uses \Cref{thm:main} to establish the dense extension in \Cref{thm:dense}.

\section{A construction on a compact set}\label{sec:compact}
We use monotonicity and countable stability of Hausdorff dimension: if $D\subset D'$, then $\dimH D\le\dimH D'$, and
\[
\dimH\Bigl(\bigcup_{i\ge1}D_i\Bigr)=\sup_{i\ge1}\dimH D_i.
\]
In particular, Hausdorff dimension is stable under finite unions and invariant under bi-Lipschitz maps. These properties are also recalled in \cite[Section 2]{LSZ2026}. For a nonempty compact set $K\subset\R$, the product identity $\dimH(K\times\R)=\dimH K+1$ and the inclusion $G_u(K)\subset K\times\R$ give
\begin{equation}\label{eq:upper}
\dimH G_u(K)\le\dimH(K\times\R)=\dimH K+1
\qquad(u\in C(K)).
\end{equation}
We also use the theorem of Bayart and Heurteaux \cite[Theorem 3]{BH}: if $L\subset\R$ is compact and $\dimH L>0$, there is $q\in C(L)$ satisfying $\dimH G_q(L)=\dimH L+1$.

\begin{lemma}\label{lem:c0}
Let $K\subset[0,1]$ be compact and let $d=\dimH K>0$. There is a linear isometry $\Phi:c_0\to C(K)$ such that
\[\dimH G_{\Phi(a)}(K)=d+1\qquad(a\in c_0\setminus\{0\}).\]
Its range is complemented in $C(K)$ by a projection of norm at most $2$.
\end{lemma}
\begin{proof}
There is $x_*\in K$ such that
\begin{equation}\label{eq:local}
\dimH\bigl(K\cap(x_*-r,x_*+r)\bigr)=d\qquad\text{for every }r>0.
\end{equation}
Indeed, if no such point existed, every point of $K$ would have a neighborhood whose intersection with $K$ has dimension smaller than $d$. Compactness would give a finite subcover, and finite stability would imply $\dimH K<d$.

Choose positive numbers $r_m<1/m$, open intervals $J_m$ with pairwise disjoint closures, and compact sets $L_m\subset K\cap J_m$ satisfying
\begin{equation}\label{eq:pieces}
d-\frac{d}{m+1}<d_m:=\dimH L_m\le d,
\qquad
\overline{J_m}\subset(x_*-r_m,x_*+r_m)\setminus\{x_*\}.
\end{equation}
To make this choice, suppose that $J_1,\ldots,J_{m-1}$ have already been chosen. Their closures are at positive distance from $x_*$, so we can choose $r_m$ as above with $(x_*-r_m,x_*+r_m)$ disjoint from all of them. By \eqref{eq:local} and $d>0$, the set
\[K\cap\bigl((x_*-r_m,x_*+r_m)\setminus\{x_*\}\bigr)\]
has dimension $d$, since removing the singleton $\{x_*\}$, of dimension zero, does not change the positive dimension $d$. Cover the punctured interval by countably many compact intervals contained in it. Countable stability supplies one of these intervals, denoted by $D_m$, for which $\dimH(K\cap D_m)>d-d/(m+1)$. Set $L_m=K\cap D_m$, and choose an open interval $J_m$ containing $D_m$ whose closure remains in the punctured interval. This proves \eqref{eq:pieces}. In particular, $d_m\to d$, and any sequence $z_m\in\overline{J_m}$ converges to $x_*$.

By the theorem of Bayart and Heurteaux, choose $q_m\in C(L_m)$ with $\dimH G_{q_m}(L_m)=d_m+1$. Since $d_m>0$, the function $q_m$ is nonzero: otherwise its graph would have dimension $d_m$. Compactness ensures that $|q_m|$ attains its positive maximum. After normalization and, if necessary, a change of sign, we may assume that $\|q_m\|_\infty=1$ and that $q_m$ takes the value $1$. Prescribe $q_m$ on $L_m$ and zero on $K\setminus J_m$. These disjoint closed subsets of $K$ carry a continuous function with values in $[-1,1]$. The bounded form of Tietze's extension theorem supplies an extension $v_m\in C(K)$ with values in $[-1,1]$. Since $v_m$ takes the value $1$ on $L_m$, we have
\begin{equation}\label{eq:vm}
v_m|_{L_m}=q_m,\qquad
\|v_m\|_\infty=1,\qquad
v_m=0\text{ on }K\setminus J_m.
\end{equation}
For $v\in C(K)$, its support is $\operatorname{supp}_K v=\overline{\{x\in K:v(x)\ne0\}}^{\,K}$, where the closure is taken in $K$. Thus the supports of the $v_m$ are contained in the pairwise disjoint sets $K\cap\overline{J_m}$.

For $j,k\in\N$, put $\ell(j,k)=2^{j-1}(2k-1)$. Every positive integer has a unique factorization as a power of $2$ times an odd positive integer. Hence the sets $\{\ell(j,k):k\in\N\}$, indexed by $j$, partition $\N$, and $\ell(j,1)\to\infty$ as $j\to\infty$. For each fixed $j$, we also have $\ell(j,k)\to\infty$ as $k\to\infty$. Define
\begin{equation}\label{eq:ej}
e_j=\sum_{k=1}^{\infty}2^{1-k}v_{\ell(j,k)}.
\end{equation}
For integers $t>r\ge0$, disjointness and \eqref{eq:vm} give
\[\left\|\sum_{k=r+1}^{t}2^{1-k}v_{\ell(j,k)}\right\|_\infty=2^{-r}.
\]
Thus the series in \eqref{eq:ej} converges uniformly, $e_j\in C(K)$, and $\|e_j\|_\infty=1$. The same disjointness gives
\begin{equation}\label{eq:pointwise}
e_i(x)e_j(x)=0\quad(x\in K,\ i\ne j),
\qquad e_j(x_*)=0\quad(j\in\N).
\end{equation}

Let $a=(a_j)\in c_0$. For integers $t>r\ge0$, equations \eqref{eq:vm}--\eqref{eq:pointwise} imply
\begin{equation}\label{eq:norm-finite}
\left\|\sum_{j=r+1}^{t}a_je_j\right\|_\infty
=\max_{r<j\le t}|a_j|.
\end{equation}
For the upper bound, at most one summand is nonzero at each point of $K$. For the reverse inequality, each $e_j$ takes the value $1$ at a point of $L_{\ell(j,1)}$, where all the other functions vanish. Since $a_j\to0$, the partial sums are uniformly Cauchy. We may therefore define
\[\Phi(a)=\sum_{j=1}^{\infty}a_je_j\in C(K).\]
Taking limits in \eqref{eq:norm-finite}, with $r=0$, gives $\|\Phi(a)\|_\infty=\|a\|_\infty$. Hence $\Phi$ is a linear isometry and $M:=\Phi(c_0)$ is closed.

To prove the dimension identity, let $a\ne0$ and choose $j_0$ with $a_{j_0}\ne0$. On $L_{\ell(j_0,k)}$, equations \eqref{eq:vm} and \eqref{eq:ej} give
\[\Phi(a)=a_{j_0}2^{1-k}q_{\ell(j_0,k)}.\]
For each fixed $k$, the scalar $a_{j_0}2^{1-k}$ is nonzero, so the vertical rescaling $(x,y)\mapsto(x,a_{j_0}2^{1-k}y)$ is bi-Lipschitz. Its constants may depend on $k$, which does not affect the dimension identity on each individual set. Thus
\begin{align*}
\dimH G_{\Phi(a)}(K)
&\ge\dimH G_{\Phi(a)}\bigl(L_{\ell(j_0,k)}\bigr)\\
&=\dimH G_{q_{\ell(j_0,k)}}\bigl(L_{\ell(j_0,k)}\bigr)
=d_{\ell(j_0,k)}+1.
\end{align*}
Since $\ell(j_0,k)\to\infty$ as $k\to\infty$, equation \eqref{eq:pieces} gives the lower bound $d+1$. The upper bound follows from \eqref{eq:upper}.

Finally, choose $t_j\in L_{\ell(j,1)}$ with $q_{\ell(j,1)}(t_j)=1$. Then $t_j\to x_*$, $e_i(t_j)=\delta_{ij}$, and \eqref{eq:pointwise} gives $\Phi(a)(x_*)=0$. Define
\[B:C(K)\longrightarrow c_0,\qquad
Bf=(f(t_j)-f(x_*))_{j\ge1}.\]
Continuity of $f$ shows that $Bf\in c_0$, and
\[\|Bf\|_\infty\le2\|f\|_\infty,
\qquad B\Phi(a)=a.\]
It follows that $P_K:=\Phi B:C(K)\to M$ is bounded, $\|P_K\|\le2$, and $P_K|_M$ is the identity. Moreover, $P_K^2=\Phi B\Phi B=\Phi B=P_K$, so $P_K$ is the required projection.
\end{proof}

\section{Proofs of the spaceability results}\label{sec:spaceable}
The following localization is the case $m=1$ of \cite[Proposition 3.10]{LSZ2026}. We include the reduction to that statement, together with the zero-dimensional case.

\begin{lemma}\label{lem:localization}
Let $1<s\le2$ and let $E\subset C[0,1]$ be finite-dimensional, with $E\setminus\{0\}\subset\Hs[0,1]$. There is a compact set $K\subset[0,1]$ such that $\dimH K=s-1$ and
\begin{equation}\label{eq:outside}
\dimH G_h([0,1]\setminus K)=s
\qquad(h\in E\setminus\{0\}).
\end{equation}
\end{lemma}
\begin{proof}
If $E=\{0\}$, take any compact set of Hausdorff dimension $s-1$. For $s<2$ a Cantor set of that dimension suffices, and for $s=2$ take $K=[0,1]$. There is no nonzero $h$ to check in \eqref{eq:outside}.

If $\dim E=n\ge1$, choose a basis $\{f_1,\ldots,f_n\}$ of $E$. In \cite[Proposition 3.10]{LSZ2026}, use the ambient compact set $[0,1]$ and $m=1$. Its measure hypothesis holds because $\mathcal H^{s-1}([0,1])=+\infty$ for $1<s<2$, whereas $\mathcal H^1([0,1])=1$ under our normalization. Its span hypothesis is the assumption on $E$. The resulting compact set has dimension $s-1$, and property~(3) of that proposition is \eqref{eq:outside}.
\end{proof}

\begin{proof}[Proof of \Cref{thm:main}]
Set $d=s-1$ and choose $K$ by \cref{lem:localization}. Apply \cref{lem:c0} to obtain $M=\Phi(c_0)\subset C(K)$. Let
\[A:C(K)\longrightarrow C[0,1]\]
be the following affine extension operator. The gaps are the connected components of the relatively open set $[0,1]\setminus K$. Each bounded interior gap is an interval $(a,b)$ with $a,b\in K$, and the possible exterior components meet $0$ or $1$ without having both endpoints in $K$. Set $Au=u$ on $K$ and
\[
(Au)(x)=\frac{b-x}{b-a}u(a)+\frac{x-a}{b-a}u(b)
\qquad(x\in(a,b))
\]
on each interior gap. On the left exterior component set $Au=u(\min K)$, and on the right exterior component set $Au=u(\max K)$.

This construction is linear and continuous on each gap and at its endpoints. To check continuity at $x\in K$, consider a sequence tending to $x$ through gaps. On any fixed gap whose closure contains $x$, the formula gives convergence to $u(x)$. Along pairwise distinct gaps, their lengths tend to zero, since they are disjoint subintervals of $[0,1]$; hence both endpoints tend to $x$. Uniform continuity of $u$ on $K$ and the convex combination in the formula again give convergence to $u(x)$. The constant exterior extensions have the same property. Thus $Au\in C[0,1]$. On every gap its values lie between the endpoint values, so $\|Au\|_\infty=\|u\|_\infty$. Consequently, $A$ is a linear isometry, and $N:=A(M)$ is closed because $M$ is complete. Set $F=E+N$.

Let $\mathcal J$ denote the family of components of $[0,1]\setminus K$. This family is countable, since each component contains a rational point and distinct components are disjoint. For $g\in N$ and $J\in\mathcal J$, write $g(x)=b_Jx+c_J$ on $J$. The affine map
\[T_J:\R^2\longrightarrow\R^2,\qquad
T_J(x,y)=(x,y+b_Jx+c_J)\]
has the Lipschitz inverse $T_J^{-1}(x,y)=(x,y-b_Jx-c_J)$, so it is bi-Lipschitz and maps $G_h(J)$ onto $G_{h+g}(J)$. The Lipschitz constants may depend on $J$; countable stability only requires preservation of dimension on each component separately. Therefore, for $h\in E\setminus\{0\}$, countable stability and \eqref{eq:outside} give
\begin{align*}
\dimH G_{h+g}([0,1]\setminus K)
&=\sup_{J\in\mathcal J}\dimH G_{h+g}(J)\\
&=\sup_{J\in\mathcal J}\dimH G_h(J)
=s.
\end{align*}
On $K$, equation \eqref{eq:upper} gives $\dimH G_{h+g}(K)\le d+1=s$. By finite stability, $h+g\in\Hs[0,1]$.

If $h=0$ and $g\in N\setminus\{0\}$, write $g=Au$ with $u\in M\setminus\{0\}$. By \cref{lem:c0}, $\dimH G_g(K)=\dimH G_u(K)=s$. On each $J\in\mathcal J$, the graph of $g$ has dimension at most $1$. Countable stability gives $\dimH G_g([0,1]\setminus K)\le1$, and finite stability again gives $\dimH G_g([0,1])=s$. Consequently,
\begin{equation}\label{eq:Fgood}
F\setminus\{0\}\subset\Hs[0,1].
\end{equation}

Every element of $N$ has graph dimension at most $1$ outside $K$. By \eqref{eq:outside} and $s>1$, no nonzero element of $E$ can belong to $N$. Thus $E\cap N=\{0\}$ and $F=E\oplus N$. The sum of a closed subspace and a finite-dimensional subspace is closed. The addition map $E\times N\to F$, $(h,g)\mapsto h+g$, is a bounded linear bijection between Banach spaces; by the bounded inverse theorem, this direct sum is topological. Since $N\simeq c_0$ and $E$ is finite-dimensional, we obtain
\[F\simeq E\oplus c_0\simeq c_0.\]
The last isomorphism follows by identifying $E$ with $\R^{\dim E}$ and adjoining these finitely many coordinates to a null sequence.

Let $R:C[0,1]\to C(K)$ be restriction, and let $P_K$ be the projection in \cref{lem:c0}. The operator $P_N:=AP_KR:C[0,1]\to N$ is bounded. For $g=Au\in N$, the identities $RAu=u$ and $P_Ku=u$ give $P_Ng=g$. Hence $P_N$ is a projection onto $N$.

Write $I$ for the identity operator on $C[0,1]$ and $\ker P_N=\{f\in C[0,1]:P_Nf=0\}$. The space $W:=(I-P_N)E$ is finite-dimensional and contained in $\ker P_N$. To construct a bounded projection $Q:\ker P_N\to W$, choose a basis $w_1,\ldots,w_r$ of $W$ and its coordinate functionals $\lambda_1,\ldots,\lambda_r$. These functionals are bounded because $W$ is finite-dimensional. By Hahn--Banach, extend them to bounded functionals $\widetilde\lambda_i$ on $\ker P_N$, and set $Qz=\sum_{i=1}^r\widetilde\lambda_i(z)w_i$. Then $Q|_W=I_W$. If $W=\{0\}$, take $Q=0$. We have $F=N\oplus W$, because each $h\in E$ decomposes as $P_Nh+(I-P_N)h$, and $N\cap W=\{0\}$. Define
\[\Pi:C[0,1]\longrightarrow F,\qquad
\Pi f=P_Nf+Q((I-P_N)f).\]
This is a bounded linear operator. For $n\in N$ and $w\in W$, it satisfies $\Pi(n+w)=n+w$. Thus it is a projection onto $F$, completing the proof.
\end{proof}

The obstruction to closed extensions in \Cref{thm:threshold} can be seen directly from the existence of an admissible dense subspace. This also explains why spaceability is compatible with failure of the extension property for infinite-dimensional initial subspaces.

\begin{proof}[Proof of \Cref{thm:threshold}]
Applying \Cref{thm:main} with $E=\{0\}$ gives a closed copy of $c_0$ in $\Hs[0,1]\cup\{0\}$. Since the algebraic dimension of $c_0$ is $\cc$, the existence requirement in the definition of $(\alpha,\cc)$-spaceability holds for every finite $\alpha$. The extension property follows from \Cref{thm:main}.

For the negative assertion, \cite[Theorem 2.6]{BMPS2021} supplies a dense subspace $M_0\subset C[0,1]$ of dimension $\cc$ such that $M_0\setminus\{0\}\subset\Hs[0,1]$. Indeed, that theorem imposes both Hausdorff and box dimensions equal to $s$, so its conclusion is stronger than needed here. Since $C[0,1]$ has a countable base and $M_0$ is dense, choosing one point of $M_0$ in each nonempty basic open set gives a countable dense set $\{u_j:j\in\N\}\subset M_0$. Set $D=\operatorname{span}\{u_j:j\in\N\}$. Then $D$ is dense and $\dim D=\aleph_0$: its dimension is at most countable and cannot be finite, because finite-dimensional subspaces of $C[0,1]$ are closed.

Fix $\aleph_0\le\alpha\le\cc$. Extend a Hamel basis of $D$ to one of $M_0$, and choose a subset of the latter basis of cardinality $\alpha$ containing the basis of $D$. Its span $E_\alpha$ satisfies
\[
D\subset E_\alpha\subset M_0,\qquad
\dim E_\alpha=\alpha,\qquad
E_\alpha\setminus\{0\}\subset\Hs[0,1].
\]
Any closed subspace of $C[0,1]$ containing $E_\alpha$ must contain $\overline D=C[0,1]$, and hence must equal $C[0,1]$. Such a subspace cannot have all its nonzero elements in $\Hs[0,1]$, since a nonzero constant function has graph dimension $1<s$. Therefore the required closed extension fails for every infinite $\alpha$, including $\alpha=\cc$.
\end{proof}

The construction in \Cref{lem:c0} uses countable stability of Hausdorff dimension, which fails for lower box dimension. It therefore does not settle \cite[Question 3.20]{LSZ2026}.

\section{Dense extensions of finite-dimensional subspaces}\label{sec:dense}
Having established the closed extensions, we now turn to density in the ambient space. The construction below preserves the initial subspace exactly and uses polynomial perturbations to make the extended subspace dense.

\begin{proof}[Proof of \Cref{thm:dense}]
Let $\mathcal P$ be the vector space of real polynomials restricted to $[0,1]$. For $p\in\mathcal P$, the map
\[
S_p:[0,1]\times\R\longrightarrow[0,1]\times\R,
\qquad S_p(x,y)=(x,y+p(x)),
\]
is bi-Lipschitz: $p$ is Lipschitz on $[0,1]$, and $S_p^{-1}=S_{-p}$ has the same property. Since $S_p$ maps $G_f([0,1])$ onto $G_{f+p}([0,1])$, invariance of Hausdorff dimension gives
\begin{equation}\label{eq:polynomial-perturbation}
\dimH G_{f+p}([0,1])=\dimH G_f([0,1])
\qquad(f\in C[0,1],\ p\in\mathcal P).
\end{equation}

By \Cref{thm:main}, choose a closed subspace $F\supset E$ with $F\simeq c_0$ and $F\setminus\{0\}\subset\Hs[0,1]$. In particular, $\dim F=\cc$. Every polynomial has a graph of Hausdorff dimension $1$, because its graph parametrization $x\mapsto(x,p(x))$ is bi-Lipschitz onto its image. As $s>1$, it follows that
\begin{equation}\label{eq:polynomial-disjointness}
F\cap\mathcal P=\{0\}.
\end{equation}

The polynomials with rational coefficients form a countable dense subset of $C[0,1]$, by the Weierstrass approximation theorem and approximation of finitely many coefficients by rational numbers. Choose a sequence $(p_n)_{n\ge1}$ in which every such polynomial occurs infinitely often. Since $F$ is infinite-dimensional and $E$ is finite-dimensional, we can inductively choose
\[
h_n\in F\setminus\bigl(E+\operatorname{span}\{h_1,\ldots,h_{n-1}\}\bigr),
\qquad \|h_n\|_\infty\le\frac1n.
\]
Indeed, at each step the excluded subspace is finite-dimensional, and a nonzero vector outside it can be scaled to satisfy the norm bound.

Let $\mathcal B_E$ be a basis of $E$ (empty if $E=\{0\}$), and extend the linearly independent set $\mathcal B_E\cup\{h_n:n\ge1\}$ to a Hamel basis $\mathcal B$ of $F$. Define an algebraic linear map $T:F\to C[0,1]$ by
\[
T(h_n)=h_n+p_n\quad(n\ge1),
\qquad T(b)=b\quad\bigl(b\in\mathcal B\setminus\{h_n:n\ge1\}\bigr).
\]
Every vector has a finite expansion in $\mathcal B$, so this defines $T$ on all of $F$. No continuity of $T$ is required. We have
\[T(v)-v\in\mathcal P\qquad(v\in F).\]
If $T(v)=0$, then $v\in F\cap\mathcal P$, and \eqref{eq:polynomial-disjointness} gives $v=0$. Thus $T$ is injective. Moreover, $T$ fixes $E$, and \eqref{eq:polynomial-perturbation} shows that $T(v)\in\Hs[0,1]$ whenever $v\ne0$. Hence $M:=T(F)$ satisfies
\[E\subset M,\qquad \dim M=\cc,\qquad M\setminus\{0\}\subset\Hs[0,1].\]

To verify density, fix a polynomial $p$ with rational coefficients and choose increasing indices $n_k$ with $p_{n_k}=p$. Then
\[T(h_{n_k})=h_{n_k}+p\longrightarrow p
\quad\text{uniformly on }[0,1],\]
because $\|h_{n_k}\|_\infty\le1/n_k$. Therefore $\overline M$ contains all polynomials with rational coefficients, and $\overline M=C[0,1]$.

For each $n\in\N$, applying \Cref{thm:main} to the zero subspace supplies a copy of $c_0$ whose nonzero elements belong to $\Hs[0,1]$, and hence an admissible $n$-dimensional initial subspace. The dense extension just proved gives $(n,\cc)$-dense-lineability. Taking $E=\operatorname{span}\{f\}$ for each $f\in\Hs[0,1]$ yields pointwise $\cc$-dense-lineability. Finally, $\dim C[0,1]=\cc$: the copy of $c_0$ gives the lower bound, while continuous functions are determined by their values on a countable dense set, giving $|C[0,1]|\le\cc$. Thus the pointwise dense-lineability is maximal.
\end{proof}

\medskip
\noindent\textbf{Funding.} This research was partially supported by CNPq (Brazil) under grants 406457/2023-9 (CNPq/MCTI Call 10/2023) and 403964/2024-5 (MCTI/CNPq Call 16/2024). Daniel M. Pellegrino also acknowledges CNPq grant 305807/2025-0, and Anselmo Raposo Jr. acknowledges CNPq grant 302341/2025-0.

\medskip
\noindent\textbf{Data availability.} No datasets were generated or analyzed in this study.

\medskip
\noindent\textbf{Code availability.} Not applicable.

\medskip
\noindent\textbf{Declaration of generative AI and AI-assisted technologies.} During the preparation of this manuscript, the authors used OpenAI's ChatGPT as a generative-AI tool for exploratory calculations, consistency checks, organization of arguments, and editorial assistance. All AI-assisted material retained in the manuscript was independently reviewed and verified by the authors, who assume full responsibility for the accuracy, originality, and integrity of the work.

\medskip
\noindent\textbf{Competing interests.} The authors declare that they have no competing interests.


\begin{thebibliography}{99}
\bibitem{AGS2005}
R. Aron, V. I. Gurariy and J. B. Seoane-Sep\'ulveda,
\newblock Lineability and spaceability of sets of functions on $\mathbb R$,
\newblock \emph{Proc. Amer. Math. Soc.} \textbf{133} (2005), no. 3, 795--803.

\bibitem{BH}
F. Bayart and Y. Heurteaux,
\newblock On the Hausdorff dimension of graphs of prevalent continuous functions on compact sets,
\newblock in \emph{Further Developments in Fractals and Related Fields}, Trends Math., Birkh\"auser/Springer, 2013, pp. 25--34.

\bibitem{BMPS2021}
A. Bonilla, G. A. Mu\~noz-Fern\'andez, J. A. Prado-Bassas and J. B. Seoane-Sep\'ulveda,
\newblock Hausdorff and Box dimensions of continuous functions and lineability,
\newblock \emph{Linear Multilinear Algebra} \textbf{69} (2021), no. 4, 593--606.

\bibitem{CMGP2024}
M. C. Calder\'on-Moreno, P. J. Gerlach-Mena and J. A. Prado-Bassas,
\newblock Infinite pointwise lineability: general criteria and applications,
\newblock \emph{Rev. R. Acad. Cienc. Exactas F\'is. Nat. Ser. A Mat. RACSAM} \textbf{118} (2024), art. 25.

\bibitem{CFST2022}
J. Carmona Tapia, J. Fern\'andez-S\'anchez, J. B. Seoane-Sep\'ulveda and W. Trutschnig,
\newblock Lineability, spaceability, and latticeability of subsets of $C([0,1])$ and Sobolev spaces,
\newblock \emph{Rev. R. Acad. Cienc. Exactas F\'is. Nat. Ser. A Mat. RACSAM} \textbf{116} (2022), art. 113.

\bibitem{FPRR2024}
V. V. F\'avaro, D. Pellegrino, A. B. Raposo Jr. and G. S. Ribeiro,
\newblock General criteria for a stronger notion of lineability,
\newblock \emph{Proc. Amer. Math. Soc.} \textbf{152} (2024), no. 3, 941--954.

\bibitem{FPT2020}
V. V. F\'avaro, D. Pellegrino and D. Tomaz,
\newblock Lineability and spaceability: a new approach,
\newblock \emph{Bull. Braz. Math. Soc. (N.S.)} \textbf{51} (2020), no. 1, 27--46.

\bibitem{Gurariy1966}
V. I. Gurariy,
\newblock Subspaces and bases in spaces of continuous functions,
\newblock \emph{Dokl. Akad. Nauk SSSR} \textbf{167} (1966), no. 5, 971--973 (in Russian).

\bibitem{LSZ2026}
J. Liu, S. Shi and Z. Zhang,
\newblock On strong spaceability of continuous functions and fractal dimensions,
\newblock arXiv:2605.25037v2, 2026.

\bibitem{LZS2025}
J. Liu, Y. Zhang and S. Shi,
\newblock On the fractal dimensions of continuous functions and algebraic genericity,
\newblock \emph{J. Math. Anal. Appl.} \textbf{546} (2025), no. 2, 129234.

\bibitem{PR2021}
D. Pellegrino and A. Raposo Jr.,
\newblock Pointwise lineability in sequence spaces,
\newblock \emph{Indag. Math. (N.S.)} \textbf{32} (2021), no. 2, 536--546.

\bibitem{Sobczyk}
A. Sobczyk,
\newblock Projection of the space $(m)$ on its subspace $(c_0)$,
\newblock \emph{Bull. Amer. Math. Soc.} \textbf{47} (1941), no. 12, 938--947.
\end{thebibliography}
\end{document}